\documentclass[12pt,a4paper]{article}
\usepackage[utf8]{inputenc}
\usepackage{tikz}
\usetikzlibrary{arrows,automata}
\tikzset{every state/.style={minimum size=0pt}}
\usepackage{smartdiagram}  
\usepackage{algorithm}
\usepackage{algcompatible}
\usepackage{graphicx}
\usepackage{xcolor}
\usepackage[algo2e,linesnumbered,ruled,vlined]{algorithm2e}
\usepackage[english]{babel}
\usepackage{subfig}
\usepackage{animate}
\usepackage{authblk}
\usepackage[margin=1in]{geometry}
\usepackage[T1]{fontenc}
\usepackage{amssymb, amsthm}
\usepackage{amsmath}
\usepackage{mathtools}
\usepackage{relsize}
\usepackage{todonotes}
\usepackage{hyperref}

\newtheorem{theorem}{Theorem}
\newtheorem{lemma}[theorem]{Lemma}

\newtheorem{openquestion}[theorem]{Open Question}
\newtheorem{observation}[theorem]{Observation}
\newtheorem{proposition}[theorem]{Proposition}
\newtheorem{example}[theorem]{Example}
\newtheorem{definition}[theorem]{Definition}
\newtheorem{claim}[theorem]{Claim}

\newcommand{\rogers}[1]{{\color{black} #1}}
\newcommand{\review}[1]{{\color{black} #1}}

\begin{document}
\title{\textbf{Extremal Results on Conflict-free Coloring\thanks{We note that one of the results in this submission, Theorem \ref{thm:max_degree}, had already appeared as part of \cite{bhyravarapu2022conflict} in the conference 47th International Symposium on Mathematical Foundations of Computer Science (MFCS 2022). The article \cite{bhyravarapu2022conflict} was authored by a subset of authors of this submission.}} \\ \vspace{0.1in}}

\author{Sriram Bhyravarapu$^1$, Shiwali Gupta$^2$, Subrahmanyam Kalyanasundaram$^3$\thanks{The third author wishes to acknowledge SERB-DST for supporting this work via grants MTR/2020/000497 and CRG/2022/009400.}, and Rogers Mathew$^4$}
\affil
{The Institute of Mathematical Sciences, HBNI, Chennai$^1$.
	\authorcr
	Department of Computer Science and Engineering, Indian Institute of Technology Hyderabad$^{2, 3, 4}$. \authorcr
	sriramb@imsc.res.in, \{cs21resch11002, subruk, rogers\}@iith.ac.in
}

\date{}
\maketitle
\textbf{keywords:}{ Conflict-free coloring, Extremal graph theory, Maximum degree, Minimum degree, Claw-free graphs} 

\begin{abstract}
A  conflict-free open neighborhood coloring  of a graph is an assignment of colors to the vertices
such that for every vertex there is a color that appears exactly once in its open neighborhood. For a graph $G$, the smallest number of colors required for such a coloring is called the conflict-free open neighborhood (CFON) chromatic number and is denoted by $\chi_{ON}(G)$. By considering closed neighborhood instead of open 
neighborhood, we obtain the analogous notions of conflict-free closed neighborhood 
 (CFCN) coloring, and CFCN chromatic number (denoted by $\chi_{CN}(G)$). The notion of conflict-free coloring was introduced in 2002, and has since received considerable attention. 

 We study CFON and CFCN colorings   
 and show the following results. In what follows, $\Delta$ denotes the maximum degree of the graph. 
 \begin{itemize}
    \item We show that if $G$ is a $K_{1,k}$-free graph then $\chi_{ON}(G)  = O(k \ln \Delta)$. D\k ebski and Przyby\l{}o in 
[JGT 2021] 
had shown that if $G$ is a line graph, then  $\chi_{CN}(G)  = O(\ln \Delta)$. As an open question, they had 
asked if their result could be extended to claw-free ($K_{1,3}$-free) graphs, which is a superclass of line graphs. Since $\chi_{CN}(G) \leq 2\chi_{ON}(G)$, our result answers their open question. It is known that there exists a
separate family of $K_{1.k}$-free graphs with $\chi_{ON}(G) = \Omega(\ln\Delta)$ and $\chi_{ON}(G) = \Omega(k)$.


     \item 

    Let  $\delta\geq 0$ be an integer. 
     We define $f_{CN}(\delta)$  as follows:  $$f_{CN}(\delta) = \max\{\chi_{CN} (G) : G \mbox{ is a graph with minimum degree }  \delta\}.$$ 

     It is easy to see that $f_{CN}(\delta') \geq f_{CN}(\delta)$ when 
     $\delta' < \delta$. Let $c$ be a positive constant.     
     It was shown [D\k{e}bski and Przyby\l o, JGT 2021] that $f_{CN}(c \Delta) = \Theta(\ln \Delta)$.
     In this paper, we show  (i)
         $f_{CN}(\frac{c\Delta}{\ln^{\epsilon} \Delta}) = O(\ln^{1+\epsilon}\Delta)$, where $\epsilon$ is a constant such that $0 \leq \epsilon \leq 1$  and 
         (ii) $f_{CN}(c\Delta^{1 - \epsilon}) = \Omega (\ln^2 \Delta)$, where $\epsilon$ is a constant such that $0 < \epsilon < 0.003$. Together with the known [Bhyravarapu, Kalyanasundaram and Mathew, JGT 2021]
         upper bound $\chi_{CN}(G) = O(\ln^2 \Delta)$, this implies that $f_{CN}(c\Delta^{1 - \epsilon}) = \Theta (\ln^2 \Delta)$.

      \item For a $K_{1, k}$-free graph $G$ on $n$ vertices, 
      we show that $\chi_{CN}(G) = O(\ln k \ln n)$. This bound is asymptotically tight since there are graphs $G$ with 
      $\chi_{CN}(G) = \Omega(\ln^2 n)$ [Glebov, Szab\'o, Tardos, CPC 2014].
 \end{itemize}

\end{abstract}

\section{Introduction}
For a hypergraph $\mathcal{H} = (V, \mathcal{E})$ and a  positive integer $k$,  a coloring $f~:~V \rightarrow [k]$ is a \emph{conflict-free coloring} (or \emph{CF coloring}) of $\mathcal{H}$ if for every $E \in \mathcal{E}$, some vertex in $E$ gets a color that is different from the color received by every other vertex in $E$. The minimum $k$ such that $f~:~V \rightarrow [k]$ is a CF coloring of $\mathcal{H}$ is called the \emph{Conflict-Free chromatic number} (or \emph{CF chromatic number}) of $\mathcal{H}$. We shall use $\chi_{CF}(\mathcal{H})$ to denote the CF chromatic number of $\mathcal{H}$.  The notion of CF coloring has been extensively studied in the context of `neighborhood hypergraphs' of graphs. Let $G$ be a graph with vertex set $V(G)$ and edge set $E(G)$. For a vertex $v \in V(G)$, the set of neighbors of $v$ in $G$ is called the \emph{open neighborhood} of $v$. We use $N_G(v)$ to denote this. The \emph{closed neighborhood} of $v$, denoted by $N_G[v]$, is $\{v\} \cup N_G(v)$.  

\begin{definition}[Conflict-free open neighborhood chromatic number]
A conflict-free coloring concerning the open neighborhoods of $G$  is an assignment of colors to $V(G)$ such that every vertex has a uniquely colored vertex in its open neighborhood. We call such a coloring a \emph{Conflict-Free Open Neighborhood coloring} (or \emph{CFON coloring}). The minimum number of colors required for a CFON coloring of $G$ is called the \emph{Conflict-Free Open Neighborhood chromatic number} (or \emph{CFON chromatic number}), denoted by $\chi_{ON}(G)$.  
\end{definition}

\begin{definition}[Conflict-free closed neighborhood chromatic number]
A conflict-free coloring concerning the closed neighborhoods of $G$  is an assignment of colors to $V(G)$ such that every vertex has a uniquely colored vertex in its closed neighborhood. We call such a coloring a \emph{Conflict-Free Closed Neighborhood coloring} (or \emph{CFCN coloring}). The minimum number of colors required for a CFCN coloring of $G$ is called the \emph{Conflict-Free Closed Neighborhood chromatic number} (or \emph{CFCN chromatic number}), denoted by $\chi_{CN}(G)$.   
\end{definition}

The following result connects CFON and CFCN chromatic numbers of a graph $G$.
\begin{proposition}[Inequality 1.3 in \cite{pach2009conflict}]
\label{prop:CFCNON}
$\chi_{CN}(G) \leq 2 \chi_{ON}(G)$.
\end{proposition}
Conflict-free coloring was introduced by Even et al. \cite{ even2003conflict} in the year 2002. Since its introduction, CF coloring of hypergraphs, CFON and CFCN coloring of graphs have been extensively studied \cite{cheilaris2009conflict, glebov2014conflict, pach2009conflict, bhyravarapu2021short, dkebski2022conflict, bhyravarapu_algorithmica, LEVTOV20091521, ElbassioniM06}. The interested reader may refer to the survey by Smorodisnky \cite{smorodinsky2013conflict} on conflict-free coloring. Abel et al. \cite{Abel17} showed that it is NP-complete to determine if a planar graph has a `partial' CFCN coloring with one color (in a partial CFCN coloring, we color only a subset of the vertices such that every vertex sees a unique color in its closed neighborhood). 

Conflict-free coloring and its variants have found applications in a frequency assignment problem in cellular networks,  in battery consumption aspects of sensor networks, in RFID protocols, and in the vertex ranking (or, ordered coloring) problem which finds applications in VLSI design, operations research, etc. \cite{smorodinsky2013conflict}. Below, we give a brief explanation of a very recently discovered application of CF coloring and its variants in the Pliable Index Coding problem (or PICOD problem) (see \cite{ krishnan2021pliable, krishnan2022pliable} for details). 
Consider a communication setup with a server having all the messages available with it, a collection of receivers where each receiver only has a subset of all the messages called `side information', and a noise-free broadcast channel connecting the server with all the receivers. In a given round of communication, a receiver is `satisfied' if it receives a message that is not in its side information currently. The PICOD problem is about finding the smallest possible $\ell$-length vector  that the server needs to send to satisfy all the receivers. In Lemma $3$ of   \cite{krishnan2022pliable}, the authors first construct a hypergraph $\mathcal{H}$ corresponding to the given instance of the PICOD problem and then show the existence of a PICOD scheme of length at most $\chi_{CF}(\mathcal{H})$. The paper uses CF coloring and introduces some new variants to prove improved bounds for the optimal PICOD length. 


\section{Definitions and notations}
For a positive integer $k$, we use $[k]$ to denote the set $\{1, 2, \ldots, k\}$.
Throughout this paper, we consider only graphs that are simple, finite, and undirected. 
For a graph $G$, we use $V(G)$ to denote its vertex set and $E(G)$ to denote its edge set. In the introduction section, we had defined the open and closed neighborhoods, denoted respectively $N_G(v)$ and $N_G[v]$, for a vertex $v$ in $V(G)$. We shall use $d_G(v)$ to denote the degree of $v$ in $G$. That is, $d_G(v) = |N_G(v)|$. For a positive integer $k$, we shall use $K_{1,k}$ to denote the complete bipartite graph with $1$ vertex in one part and $k$ vertices in the other part. 
A graph is \emph{$K_{1,k}$-free} if it does not contain $K_{1,k}$ as an induced subgraph. Graphs that are $K_{1,3}$-free are also known by the name \emph{claw-free} graphs. The \emph{claw number} of a graph $G$
is defined to be the largest $k$ for which $G$ contains $K_{1, k}$ as an induced subgraph.

Given a graph $G$, the \emph{line graph} of $G$, denoted by $L(G)$, is the
graph defined as follows: The vertex set of $L(G)$ is $V(L(G)) = E(G)$ and two vertices $e_1, e_2$ of $L(G)$
are adjacent to each other if and only if in the original graph $G$, the edges $e_1$ and $e_2$ share an end point. 

Given a hypergraph $\mathcal{H} = 
(V, \mathcal{E})$, the \emph{degree} of an element $v \in V$, denoted by $d_{\mathcal{H}}(v)$, is the number of hyperedges that $v$ is present in. 
The \emph{maximum degree of the hypergraph $\mathcal{H}$} is $\max \{d_\mathcal{H}(v)~:~v \in V\}$. 

\section{Our contributions and open questions}
D\k{e}bski and Przyby\l o in \cite{dkebski2022conflict} showed that for a graph $G$ with maximum degree $\Delta$, the CFCN chromatic number of its line graph is $\chi_{CN}(L(G))$ $= O(\ln \Delta)$.
Note that line graphs are a subclass of claw-free graphs (or $K_{1,3}$-free graphs). 
The following 
example 
implies that the upper bound of $O(\ln \Delta)$ from  \cite{dkebski2022conflict} is asymptotically tight. 

\begin{example}
\label{example:line_graph}
Consider $L(K_n)$, the line graph of the complete graph on $n$ vertices. In \cite{dkebski2022conflict}, it was shown that $\chi_{CN}(L(K_n)) = \Omega(\ln n)$. 
Since $\chi_{CN}(G) \le 2 \chi_{ON}(G)$
(Proposition \ref{prop:CFCNON}), this implies $\chi_{ON}(L(K_n)) = \Omega(\ln n)$.
\end{example}

\begin{example}
\label{example:subdivide}
Let $K_n^*$ be the $K_{1,n}$-free graph with maximum degree $n-1$ obtained by subdividing every edge of $K_n$ exactly once. It is known (see \cite{pach2009conflict}) that $\chi_{ON}(K_n^*) = n$. 
\end{example}


Let us first discuss the dependence of the CFON chromatic number of a graph 
on its claw number $k$ and maximum degree $\Delta$.  
Example \ref{example:line_graph} is a family of graphs whose maximum degree is 
$2n-4$ and claw number is 2. This means that $\chi_{ON}(G)$ cannot be a function 
of the form $k \cdot h(\Delta)$ where $h (\Delta) = o(\ln \Delta)$. On the
other hand, Example \ref{example:subdivide} is a family of graphs where 
$\Delta = n - 1$ and $k = n- 1$. This means that  $\chi_{ON}(G)$ cannot be a function 
of the form $g(k) \cdot \ln \Delta$, where $g(k) = O(k^{1 - \epsilon})$, for an $\epsilon > 0$. 

We complement the above observations with an upper bound of $\chi_{ON}(G) = O(k \ln \Delta)$. This implies an upper bound $\chi_{CN}(G) = O(k \ln \Delta)$ by Proposition \ref{prop:CFCNON}.
Our result, proved in Section \ref{sec:max_degree_claw}, generalizes the upper bound of $\chi_{CN}(G) = O(\ln \Delta)$ \cite{dkebski2022conflict} for line graphs. 
As mentioned before, line graphs are a subclass of claw-free graphs. 
In many of the practical applications that motivate conflict-free coloring, the 
underlying graphs happen to be geometric intersection graphs such as unit disk graphs, unit square graphs, etc. \cite{LEVTOV20091521, smorodinsky2013conflict}. These graph classes are usually $K_{1,k}$-free for some constant $k$. For instance, unit disk graphs are $K_{1,6}$-free.

It was 
posed as an open question in \cite{dkebski2022conflict} if the $O(\ln \Delta)$
upper bound could be generalized to claw-free graphs. Our $O(k \ln \Delta)$ upper bound answers this question in the affirmative.
Though Examples \ref{example:line_graph} and \ref{example:subdivide} imply the existence of graphs $G$
for which $\chi_{ON}(G) = \Omega(k)$ and $\chi_{ON}(G) = \Omega(\ln \Delta)$, 
it is of interest to know whether 
the upper bound of $O(k \ln \Delta)$  is tight.




\begin{openquestion} 
Are there $K_{1,k}$-free graphs $G$ with maximum degree $\Delta$ for which $\chi_{ON}(G) = \Omega(k \ln \Delta)$?
\end{openquestion}

In Section \ref{sec:max_degree_claw}, we show that if $G$ is a $K_{1,k}$-free graph on $n$ vertices, then $\chi_{CN}(G) = O(\ln k \ln n)$. This bound is asymptotically tight as it was shown in \cite{glebov2014conflict} that there exist graphs $G$ on $n$ vertices with $\chi_{CN}(G) = \Omega(\ln ^2 n)$. This still leaves
the possibility of the following improvement:

\begin{openquestion}
Can a bound of  $O(\ln k \ln \Delta)$ be obtained for $\chi_{CN}(G)$, for 
$K_{1,k}$-free graphs $G$ with maximum degree $\Delta$?
\end{openquestion}

Now we turn our attention to CFCN chromatic number for graphs of a specified minimum degree.
Let $\Delta$ denote the maximum degree of the graph under consideration and let $c$ be any positive constant. 
{Let $\delta\geq 0$ be an integer.
We define  $$f_{CN}(\delta) := \max \{\chi_{CN}(G)~:~G \mbox{ is a graph with minimum degree equal to } \delta\}.$$ 
It is easy to see that $f_{CN}(\delta') \geq f_{CN}(\delta)$ when $\delta' < \delta$. 
The reader may refer to the discussion at the beginning of Section \ref{sec: 1 - epsilon} for a proof. D\k{e}bski and Przyby\l o in \cite{dkebski2022conflict} showed that $f_{CN}(c \Delta) = \Theta(\ln \Delta)$. 
In Section \ref{subsec:max_degree}, we show that $f_{CN}(c \frac{\Delta}{\ln ^{\epsilon} \Delta }) = O(\ln ^{1 + \epsilon} \Delta)$, where $0 \leq \epsilon \leq 1$.
A natural open question is if this bound is tight.
\begin{openquestion}\label{oq:fcn}
Is $f_{CN}(c\frac{\Delta}{\ln ^{\epsilon} \Delta }) = \Omega(\ln ^{1 + \epsilon} \Delta)$?   
\end{openquestion}

Further, in Section \ref{sec: 1 - epsilon}, we show that $f_{CN}(c \Delta^{1 - \epsilon}) = \Omega(\ln^2 \Delta)$, for $0< \epsilon < 0.003$. 
It was shown by Bhyravarapu, Kalyanasundaram, and Mathew \cite{bhyravarapu2021short} that for any graph $G$, $\chi_{CN}(G) = O(\ln ^2 \Delta)$. Combining both, we get $f_{CN}(c \Delta^{1 - \epsilon}) = \Theta(\ln ^2 \Delta)$. 
An affirmative answer to Open Question \ref{oq:fcn} will help us understand the 
function $f_{CN}$ in its full range.

Analogous to the function $f_{CN}$, we can define a function $f_{ON}$ as 
$$f_{ON}(\delta) := \max \{\chi_{ON}(G)~:~G \mbox{ is a graph with minimum degree equal to } \delta\}.$$ 
Like in the case of CFCN coloring, we have that 
$f_{ON}(\delta') \geq f_{ON}(\delta)$ when $\delta' < \delta$. 
The results in \cite{dkebski2022conflict} imply\footnote{The article \cite{dkebski2022conflict} explicitly discusses only CFCN chromatic number. However, 
the proof techniques of the upper bound 
extend to yield an identical upper bound for CFON chromatic number. The lower bound in \cite{dkebski2022conflict} implies a similar lower bound for CFON chromatic number by an application of  Proposition
\ref{prop:CFCNON}.} 
that $f_{ON}(c \Delta) = \Theta(\ln \Delta)$. It was shown by Pach and Tardos \cite{pach2009conflict} that for any graph $G$ with minimum degree $c \log \Delta$, 
we have $\chi_{ON}(G) = O(\ln ^2 \Delta)$. 
Combining this with our result in Section \ref{sec: 1 - epsilon}, we have $f_{ON}(c \Delta^{1 - \epsilon}) = \Theta(\ln ^ 2 \Delta)$, where $0 < \epsilon < 0.003$.  What is the value of $f_{ON}(\delta)$, when $\delta = o(\ln \Delta)$?  It is known that $\chi_{ON}(G) \leq \Delta + 1$, for any graph $G$. This bound is tight as $\chi_{ON}(K_n^*) = n$. Thus, $f_{ON}(c) = \Theta(\Delta)$. This leaves us with the following open question. 

\begin{openquestion}
  What is the value of $f_{ON}(\delta)$ when $ \delta = o(\ln \Delta)$ and $\delta$ is not any absolute constant?   
\end{openquestion}

Theorem 1.2 in \cite{pach2009conflict} implies that $f_{ON}(\delta) = O(\delta \cdot \Delta^{\frac{2}{\delta}} \cdot \ln \Delta)$. However, it is not clear whether this bound is tight in the range of values of $\delta$ that we are interested in. 
\vspace{0.1in}


\section{Auxiliary results}
\label{sec:auxiliary_lemmas}
In this section, we state a few known auxiliary lemmas and theorems that will be used later.
We state the local lemma that will be used in the proof of Lemma \ref{lem_near_uniform_hypergraph} and Chernoff bound that will be used in the proof of Theorem \ref{thm:lower_bound}. 

\begin{lemma}[\emph{The Local Lemma}, \cite{lovaszlocallemma}] \label{lem:local} Let $A_1, \ldots , A_n$ be events in an arbitrary probability space. Suppose that each event $A_i$ is mutually independent of a set of all the other events $A_j$ but at most $d$, and that $Pr[A_i] \leq p$ for all $i \in [n]$. If  $4pd \leq 1$, then $Pr[\cap _{i=1}^n \overline{A_i}] > 0$.  
\end{lemma}

 \begin{theorem}[Chernoff Bound, Corollary 4.6 in \cite{mitzenmacher}]
\label{thm_Chernoff}
Let $X_1, \ldots , X_n$ be independent Poisson trials such that $Pr[X_i] = p_i$. Let $X = \sum_{i=1}^n X_i$ and $\mu = E[X]$. For $0 < \delta < 1$, 
 $Pr[|X-\mu| \geq \delta \mu] \leq 2e^{-\mu \delta^2/3}$. 
\end{theorem}

The theorem below gives an upper bound to the CF chromatic number of a hypergraph in terms of its maximum degree. 

\begin{theorem}[Theorem 1.1(b) in \cite{pach2009conflict}]
\label{thm_pach_tardos}
Let $\mathcal{H}$ be a hypergraph and let $\Delta$ be the maximum degree of any vertex in $\mathcal{H}$. Then, $\chi_{CF}(\mathcal{H}) \leq \Delta + 1$.  
\end{theorem}

Finally, we prove the following lemma for CF chromatic number of near uniform hypergraphs. We will use this lemma in the proofs of Theorems \ref{thm_claw_max_degree} 
and \ref{thm:max_degree}.

\begin{lemma} 
\label{lem_near_uniform_hypergraph}
Let $\mathcal{H} = (V,\mathcal{E})$ be a hypergraph where (i) every hyperedge intersects with at most $\Gamma$ other hyperedges, and (ii) for every hyperedge $E \in \mathcal{E}$, $r \le |E| \leq \ell r$, 
where $\ell \geq 1$ is some integer and $r \ge 2 \log_2 (4 \Gamma)$. Then, 
$\chi_{CF}(\mathcal{H}) \leq  e \ell r$, where $e$ is the base of natural logarithm. 
\end{lemma}

\begin{proof}
For each vertex in $V$, assign a color that is chosen independently and uniformly at random from a set of $e \ell r$ colors. We will first 
show that the probability of this coloring being bad for an edge is small, and then use Local Lemma to show the existence of conflict-free coloring for $\mathcal H$ using at most $e \ell r$ colors.

Consider a hyperedge $E \in \mathcal{E}$ with $m: = |{E}|$.
By assumption, we have $r \leq m \leq \ell r$. 
Let $A_E$ denote the bad event that $E$ is colored with  $\leq |E|/2$ 
colors. Note that if $A_E$ does not occur, then $E$ is colored with 
$> |E|/2$ colors, hence there is at least one color that appears exactly once in $E$.

\begin{eqnarray*}
    Pr[A_E] & \leq & \binom{e \ell r}{m/2} \left( \frac{m/2}{e \ell r}\right)^m \\
    & \leq & \left(\frac{e^2 \ell r}{m/2} \right)^{m/2}\left( \frac{m/2}{e \ell r}\right)^m \quad \quad (\mbox{since  }\binom{n}{k} \leq \left(\frac{en}{k}\right)^k) \\
    & = & \frac{\left(m/2\right)^{m/2}}{(\ell r)^{m/2}} \quad = \quad \left( \frac{m}{2 \ell r} \right)^{m/2} \\
    & \leq & \left( 1/2 \right)^{m/2} 
     \quad \leq \quad  \frac{1}{4 \Gamma}. 
\end{eqnarray*}
Here the penultimate inequality follows since $m \leq \ell r$, and the last
inequality follows since $m \geq 2 \log_2 (4 \Gamma)$.

We apply the Local Lemma (Lemma \ref{lem:local}) on the events $A_E$ for all hyperedges $E \in \mathcal E$.
Since each hyperedge intersects with at most $\Gamma$ other hyperedges, and $4\cdot \frac{1}{4\Gamma}\cdot \Gamma \leq 1$,  
we get $Pr[\cap_{E \in \mathcal{E}}(\overline A_E)] > 0$. That is, there is a conflict free coloring of $\mathcal{H}$ that uses at most $e\ell r$ colors. This completes the proof of the lemma. 
\end{proof}

\section{$K_{1,k}$-free graphs}
\label{sec:max_degree_claw}
In this section, we show improved upper bounds for CFON and CFCN chromatic numbers on
$K_{1,k}$-free graphs.

 \begin{theorem}
\label{thm_claw_max_degree}
Let $G$ be a $K_{1,k}$-free graph with maximum degree $\Delta \geq 2$ having no isolated vertices. Then, $\chi_{ON}(G) = O(k\ln \Delta)$. 
\label{klogdeltaproof}
\end{theorem}

\begin{proof}
Let $A$ be a maximal independent set of $G$. Let \review{$A_1 : = \{v \in A : d_G(v) \le 12 k \ln \Delta \}$} and $A_2 : = A \setminus A_1$. Let $X : = \bigcup_{v \in A_1} N_G(v)$. 

Next, we obtain $G'$ by removing all the vertices from $G$ that belong to $A \cup X$.
In other words, $G' = G[V \setminus (A \cup X)]$.
Since $A$ is a maximal independent set in $G$, every $v \in V \setminus (A \cup X)$ has a neighbor, say $w$, in $A$. Since no vertex in $A_1$ has a neighbor in $V(G')$, $w \in A_2$.

We start with a proper coloring of $G'$, say $h: V(G') \longrightarrow [s] = \{1, 2, \dots, s\}$  that uses at most $\Delta + 1$ colors.  Let $L_1, L_2, \ldots, L_{s}$ be the color classes with respect to the coloring $h$, where $s \leq \Delta + 1$.

\begin{observation}\label{obs:1}
For every $2 \le i \le s$, we may assume that every vertex $v \in L_i$ has a neighbor in each $L_j$, $1 \le j < i$. If $v$ has no neighbor in $L_j$, $j < i$, we can move $v$ to $L_j$.
\end{observation}


\begin{observation}\label{obs:2}
    Since $G$ is $K_{1,k}$-free,  any vertex in $G$ has at most $k - 1$ neighbors in $L_i$, for every $i \in [s] $.
\end{observation}

\begin{observation}\label{obs:3}
    Consider a subset $\widehat{A} \subseteq A$, and let $\widehat{\mathcal H} = (\widehat{V}, \widehat{\mathcal E})$ be defined as follows: $\widehat{V} = \bigcup_{v \in \widehat{A}} N_G(v)$ and $\widehat{\mathcal E} = \{ N_G(v) : v \in \widehat{A} \}$. Since $G$ is $K_{1,k}$-free and $A$ is an independent set, the maximum degree of
    $\widehat{\mathcal H}$ is at most $k-1$.
\end{observation}

 If $s > 12 \ln \Delta$, then we define $B := L_1 \cup L_2 \cup \dots \cup L_{12 \ln \Delta}$ and $C = V(G') \setminus B$. Otherwise, we define $B := L_1 \cup L_2 \cup \dots \cup L_{s}$ and $C = \emptyset$.

 We obtain the desired CFON coloring of $G$ by conflict-free coloring five hypergraphs, $\mathcal{H}_1, \dots, \mathcal{H}_5 $, which are defined below. Note that the set of colors we use to color each hypergraph $\mathcal{H}_i$ is disjoint from the set of colors we use to color any other hypergraph $\mathcal{H}_j$, $1 \le i < j \le 5$. 
\begin{itemize}
    \item 
 Suppose $C \neq \emptyset$. We define a hypergraph $\mathcal{H}_1 = (V_1, \mathcal{E}_1)$, where $V_1 = B$ and $\mathcal{E}_1 = \{N_{G'}(v) \cap B : v \in C\}$. 
\review{The following observation follows from Observations \ref{obs:1} and \ref{obs:2} 
\begin{observation}
 \label{obv:degree_in_B}
 Every vertex in $C$ has at least  $12\ln \Delta$ neighbors in $B$. Further, for every $v \in V(G)$, $|N_G(v) \cap B| \leq 12 (k-1) \ln \Delta$. 
\end{observation}} 
 So for each $E \in \mathcal E_1$, we have $12 \ln \Delta \leq |E| \leq 12 (k-1) \ln \Delta$.  
By applying Lemma \ref{lem_near_uniform_hypergraph} to $\mathcal H_1$, with $\ell = k - 1, r = 12 \ln \Delta$, and $\Gamma \le \Delta^{2}$, we get $\chi_{CF}(\mathcal{H}_1)  \le e \cdot (k -1)\cdot 12 \ln \Delta$. 

Here, $\Gamma$ denotes the number of other hyperedges a given hyperedge $E:= N_{G'}(v) \cap B$, for some $v \in C$, is overlapping with. 
Since the maximum degree is $\Delta$, and since $E \subseteq N_G(v)$, it follows that 
$\Gamma \leq \Delta^2$.
The conflict-free coloring of hypergraph $\mathcal{H}_1$ ensures that all the vertices in $C$, see a unique color in their open neighborhood.

\item 
Similarly, we define a hypergraph $\mathcal{H}_2 = (V_2, \mathcal{E}_2)$, where $V_2 = A_2$ and $\mathcal{E}_2 = \{N_{G}(v) \cap A_2 : v \in B\}$. By Observation \ref{obs:2}, the maximum degree of $\mathcal{H}_2$ is at most $ (k - 1) \cdot 12 \ln \Delta$. Hence by Theorem \ref{thm_pach_tardos}, we have  $\chi_{CF}(\mathcal{H}_2) \le (k - 1) \cdot 12\ln\Delta + 1$. This conflict-free coloring of $\mathcal{H}_2$ 
ensures that every $v \in B$ sees a unique color in its open neighborhood.

\item 
Let $\mathcal{H}_3 = (V_3, \mathcal{E}_3)$ be a hypergraph, where $V_3 = A_1$ and $\mathcal{E}_3 = \{N_G(v) \cap A_1 : v \in X \}$. 
By choice of vertices in $A_1$, the maximum degree of $\mathcal{H}_3$ is at most \review{$12k \ln \Delta$. Hence by Theorem \ref{thm_pach_tardos}, $\chi_{CF}(\mathcal{H}_3) \le 12k \ln \Delta + 1$.} The conflict-free coloring of $\mathcal{H}_3$ ensures that every $v \in X$ sees a unique color in its open neighborhood.

\item 
Now we need to handle the needs of the vertices in $A$. 
We first partition $A$ as follows: let $A_X = \{v \in A : N_G(v) \cap X \neq \emptyset \}$ and $A_{\overline{X}} = A \setminus A_X$. \review{From the  definitions of $A_1, A_2, A_X, A_{\overline{X}}$, we can say that 
$(A_1, A_2)$ and $(A_X, A_{\overline{X}})$ are two bipartitions of $A$ that satisfy (a) $A_1 \subset A_X$, and (b) $A_{\overline{X}} \subseteq A_2$. 
From the definition of $A_{\overline{X}}$, no vertex in $A_{\overline{X}}$ has a neighbor in $X$. Further, since $A_{\overline{X}} \subseteq A_2$ and every vertex in $A_2$ has degree greater than $12k\ln \Delta$, we have the following observation.
\begin{observation}
\label{obv:deg_A_X_bar}    
For every $v \in A_{\overline{X}}$, (i) $N_G(v) \subseteq B \cup C$, and (ii) $|N_G(v)| > 12k\ln \Delta$. 
\end{observation}}
We define a hypergraph $\mathcal{H}_4 = (V_4, \mathcal{E}_4)$, where $V_4 = X$ and $\mathcal{E}_4 = \{N_G(v) \cap X : v \in A_X\}$. 
By Observation \ref{obs:3},  the maximum degree of $\mathcal{H}_4$ is at most $k - 1$. 
Hence by Theorem \ref{thm_pach_tardos}, $\chi_{CF}(\mathcal{H}_4) \le k$. This coloring addresses the requirements of the vertices in $A_X$. \review{We define another hypergraph $\mathcal{H}_5$ below to address the requirements of the vertices in $A_{\overline{X}}$.} 

\item \review{Suppose $C \neq \emptyset$. Then, we construct a hypergraph $\mathcal{H}_5 = (V_5, \mathcal{E}_5)$, where $V_5 = C$ and $\mathcal{E}_5 = \{N_G(v) \cap C : v \in A_{\overline{X}}\}$. From Observations \ref{obv:degree_in_B} and \ref{obv:deg_A_X_bar}, it follows that $N_G(v) \cap C \neq \emptyset$, for every $v \in C$. By Observation \ref{obs:3}, we know that the maximum degree of $\mathcal{H}_5$ is at most $k - 1$.} 
Therefore, by Theorem \ref{thm_pach_tardos}, $\chi_{CF}(\mathcal{H}_5) \le k$. 

\review{Suppose $C = \emptyset$. Then, we claim that $A_{\overline{X}} = \emptyset$. Assume, for the sake of contradiction, that $v \in A_{\overline{X}}$. Then, by Observation \ref{obv:deg_A_X_bar}, $|N_G(v)| > 12k\ln \Delta$ and $N_G(v) \subseteq B$. By Observation \ref{obv:degree_in_B}, $|N_G(v)| \leq 12(k-1)\ln \Delta$ which is a contradiction. Hence our assumption that $A_{\overline{X}}$ is non-empty is false.}

This addresses the needs of the vertices in $A_{\overline{X}}$.
\end{itemize}
 
Note that we have addressed the needs of all the vertices in $G$. Also, each vertex is colored at most once in the above. There may be vertices that are left uncolored because
they did not feature in any of the hypergraphs. We can assign all these vertices a 
new color, obtaining a conflict-free coloring of $G$ that uses at most \review{$57 k \ln \Delta + 2k + 3$ colors.} 
\end{proof}

The following result gives an upper bound for CFCN chromatic number of $K_{1, k}$-free graphs.

 \begin{theorem}
\label{thm_claw_n}
Let $G$ be a $K_{1,k}$-free graph with $n$ vertices. Then, $\chi_{CN}(G) = O(\ln k \ln n)$.  
\end{theorem}

\begin{proof}
We first give a brief overview of the proof.
We use an approach similar to the one used by Pach and Tardos  for the proof of Theorem 1.6 in \cite{pach2009conflict}. 
Using a probabilistic approach, we show the existence of a subset $I_1^*$ of a maximal independent set. By coloring all the vertices of $I_1^*$ with the color 1, we can ensure that a $c/\log_2 k$ fraction of the vertices of the graph see a uniquely colored neighbor.  
We can repeat this process $O(\ln k \ln n)$ times to ensure that all the vertices of $G$ see a uniquely colored neighbor. 


We now describe in detail how we pick the random independent sets.
Let $G_1 = G$. Let $S_1$ be a maximal independent set in $G_1$. Pick an integer $i$ uniformly at random from the set $\{0, 1, \dots, \lfloor \log_2 k \rfloor\}$. Select $I_1 \subseteq S_1$ by picking $v$ into $I_1$ with probability $2^{-i}$ independently, for each vertex $v \in S_1$. 
Note that the integer $i$ can be equal to 0 with probability $\frac{1}{\lfloor \log_2 k \rfloor + 1}$. In that case, every $v \in S_1$ is chosen into $I_1$ with probability 1.  
\begin{equation}\label{eqn:1}
   \forall v \in S_1 , Pr[\text{$v$ is chosen into $I_1$}]  \ge  \frac{1}{\lfloor \log_2 k \rfloor + 1}\;. 
\end{equation}
Color every vertex in $I_1$ with color 1. Let $A_1 = \{v \in V(G_1) \setminus S_1 : |N_{G_1}(v) \cap I_1| = 1\}$. For a vertex $w \in V(G_1) \setminus S_1$, we define $d_w : = |N_{G_1}(w) \cap S_1|$. Note that for any $w \in V(G_1) \setminus S_1$, we have $k-1 \geq d_w \ge 1$ as (1) $G$ is $K_{1, k}$-free, and (2) $S_1$ is a maximal independent set in $G_1$. For a vertex $w \in V(G_1) \setminus S_1$, what is the probability that $w \in A_1$ $(\text{or} \hspace{0.1cm} |N_{G_1}(w) \cap I_1| = 1)$?
Let $A_{1}^w$ denote the event that $w \in A_1$. 
Let $p = \frac{1}{2^{\lfloor \log_2 d_w \rfloor}}$. 
Below, we estimate the probability of the event $A_{1}^w$.

\begin{eqnarray*}
    Pr[A_{1}^w] & = & \sum_{x=0}^{\lfloor \log_2 k \rfloor} Pr[i = x] \cdot Pr[A_{1}^w | i = x] \\
    & \ge & Pr[i = \lfloor \log_2 d_w \rfloor] \cdot Pr[A_{1}^w | i = \lfloor \log_2 d_w \rfloor]  \\
    & = & \frac{1}{\lfloor \log_2 k \rfloor + 1} \Bigg(d_w \cdot p (1 - p)^{d_w - 1}\Bigg) \;.
\end{eqnarray*}
We analyze the above expression for different values of $d_w$. When $d_w = 1$, we can check that $Pr[A_{1}^w | i = \lfloor \log_2 d_w \rfloor] = 1$. When $2 \leq d_w \leq 4$, we can verify using direct calculations that $d_w \cdot p (1 - p)^{d_w - 1} > 0.02$. For the remaining values of $d_w$, 
we have that 
$$d_w \cdot p (1 - p)^{d_w - 1} \geq \bigg(1 - \frac{2}{d_w}\bigg)^{d_w - 1} \geq \bigg(1 - \frac{2}{d_w}\bigg)^{d_w} \geq e^{-2} \bigg(1 - \frac{4}{d_w}\bigg) \geq c,$$
where $c = 0.02$. The first inequality holds since $\frac{1}{d_w} \leq p \leq \frac{2}{d_w}$, and the third inequality holds since $(1 + {x}/{n})^n \ge e^x (1 - {x^2}/{n})$ for $n \ge 1$, $|x| \le n$. The last inequality holds when $d_w \geq 5$.
Thus we get the following:
\begin{equation}\label{eqn:2}
    \forall w \in V(G_1) \backslash S_1 , Pr[\text{$w$ is chosen into $A_1$}]  \ge  \frac{c}{\lfloor \log_2 k \rfloor + 1} \;. 
\end{equation}
From equations (\ref{eqn:1}) and (\ref{eqn:2}), we have the expected cardinality of $I_1 \cup A_1$ is at least $\frac{c \cdot |V(G_1)|}{\lfloor \log_2 k \rfloor + 1}$.
By the probabilistic method, this implies the existence of $I_{1}^*$ 
such that at least $\frac{c \cdot |V(G_1)|}{\lfloor \log_2 k \rfloor + 1}$
vertices of $G_1$ see a uniquely colored neighbor. Color all the vertices of $I_{1}^*$ with color 1.

Let $G_2$ be the subgraph of $G_1$ induced on the vertices that do not have a uniquely 
colored neighbor. We can repeat the same construction and argument for $G_2$, ensuring that $\frac{c \cdot |V(G_2)|}{\lfloor \log_2 k \rfloor + 1}$
vertices of $G_2$ see a uniquely colored neighbor. 

After $r$ such rounds, there will be at most $n \big(1- \frac{c}{\log_2 k}\big)^r$ many vertices of the graph that do not have a uniquely colored neighbor. By setting $r > (\ln n \log_2 k)/c$, we get that there are $< 1$ vertices that do not have a 
uniquely colored neighbor. 
Since we use one new color per round, we need $(\ln n \log_2 k)/c = O(\ln k \ln n)$ colors. 
\end{proof}

\section{Graphs with high minimum degree}
\label{subsec:max_degree}
We first prove the following lemma, which will be used in the proof of Theorem \ref{thm:max_degree}.

\begin{lemma}
\label{lem:restricting_degree}
Let $\Delta$ denote the maximum degree of a graph $G$. It is given that every vertex in $G$ has degree at least $\frac{c\Delta}{\ln ^\epsilon \Delta}$ for some $\epsilon \geq 0$ and $c$ is a constant. Then, there exists $A \subseteq V(G)$ such that for every vertex $v \in V(G)$, 
$$108\ln(2\Delta) < |N_G(v) \cap A| < \frac{180}{c}\ln^{1+\epsilon}(2\Delta). $$
\end{lemma}
\begin{proof}
We construct a random subset $A$ of $V(G)$ as described below. Each $v \in V(G)$ is independently chosen into $A$ with probability
$\frac{144\ln^{1+\epsilon}(2\Delta)}{c\Delta}$. 
For a vertex $v \in V(G)$, let $X_v$ be a random variable that  denotes $|N_G(v) \cap A|$. Then, $\mu_v := E[X_v] = \frac{144\ln^{1+\epsilon}(2\Delta)}{c\Delta}d_G(v) \geq 144\ln  (2\Delta)$. 
Since $d_G(v) \leq \Delta$, we also have $\mu_v \leq \frac{144\ln^{1+\epsilon}(2 \Delta)}{c}$. Let $B_v$ denote the event that $|X_v - \mu_v| \geq \frac{\mu_v}{4}$. Applying Theorem \ref{thm_Chernoff} with $\delta = 1/4$, we get $Pr[B_v] = Pr[|X_v - \mu_v| \geq \frac{\mu_v}{4}] \leq 2e^{-\frac{\mu_v}{48}} \leq 2e^{-\frac{144\ln (2\Delta)}{48}} = 
\frac{2}{(2\Delta)^{3}}$.   The event $B_v$ is mutually independent of all but those events $B_u$ where $N_G(u) \cap N_G(v) \neq \emptyset$. Hence, every event $B_v$ is mutually independent of all but at most $\Delta^2$ other events. 
Applying Lemma \ref{lem:local} with $p = Pr[B_v] \leq \frac{2}{(2\Delta)^{3}}$ and $d = \Delta^2$, we have $4\cdot \frac{2}{(2\Delta)^{3}} \cdot \Delta^2 \leq 1$. Thus, there is a non-zero probability that none of the events $B_v$ occur. 
In other words, for every $v$, it is possible to have $ \frac{3}{4}\mu_v < X_v < \frac{5}{4}\mu_v$. Using the upper and lower bounds of $\mu_v$ we computed above, we can say that there exists an $A$ such that, for every $v$,    $108\ln(2 \Delta) < |N_G(v) \cap A| < \frac{180}{c}\ln^{1+\epsilon}(2\Delta)$. 
\end{proof}

The following theorem provides improved upper bounds for $\chi_{ON}$ in terms of its maximum degree for graphs $G$ that have high minimum degrees.

\begin{theorem}
\label{thm:max_degree}
Let $G$ be a graph with maximum degree $\Delta$. 
It is given that every vertex in $G$ has a degree at least
$\frac{c\Delta}{\ln^{\epsilon} \Delta}$ for some $\epsilon \geq 0$
and $c$ is a constant. Then, $\chi_{ON}= O(\ln^{1+\epsilon}\Delta)$. 
\end{theorem}
\begin{proof}
Apply Lemma \ref{lem:restricting_degree} to find an $A \subseteq V(G)$ such that for every $v \in V(G)$, $108\ln(2\Delta) < |N_G(v) \cap A| < \frac{180}{c}\ln^{1+\epsilon}(2\Delta)$. Construct a hypergraph $\mathcal{H} = (A,\mathcal{E})$ where $\mathcal{E} = \{N_G(v) \cap A~:~v \in V(G)\}$. 
Every $E \in \mathcal{E}$ satisfies $2 \log_2 (4 \Delta^2) < 108 \ln (2\Delta) < |E| < \frac{180}{c}\ln^{1+\epsilon}(2\Delta)$. 
Applying Lemma \ref{lem_near_uniform_hypergraph} with \rogers{$r = 108 \ln (2 \Delta)$ and $\ell = \frac{5}{3c}\log^{\epsilon}(2\Delta)$}, we get $\chi_{CF}(\mathcal{H}) \leq \frac{490}{c}\ln^{1 + \epsilon}(2\Delta)$. By assigning an unused color to the vertices 
in $V(G) \setminus A$, we can extend a conflict-free coloring of $\mathcal{H}$ to a CFON coloring for $G$. 
\end{proof}

\section{A lower bound}
\label{sec: 1 - epsilon}
Let $\delta \geq 0$ be an integer. Recall that, $f_{CN}(\delta) = \max\{\chi_{CN}(G) : G\mbox{ has minimum degree } \mbox{equal to }\delta \}$. We claim that $f_{CN}(\delta') \geq f_{CN}(\delta)$, when $\delta' < \delta$. Let $H$ be a graph with minimum degree equal to $\delta$ having $\chi_{CN}(H) = f_{CN}(\delta)$. Let $H'$ be any graph with a minimum degree equal to $\delta'$. Then the graph $H \oplus H'$ obtained by taking the disjoint union of $H$ and $H'$ has a minimum degree equal to $\delta'$ and $\chi_{CN}(H \oplus H') \geq f_{CN}(\delta)$. This proves the claim. As discussed earlier, with $\Delta$ denoting the maximum degree of the graph under consideration and with $c$ denoting any positive constant, we know that $f_{CN}(c\Delta) = \Theta(\ln \Delta)$. In this section, in Theorem \ref{thm:lower_bound}, we show that $f_{CN}(c\Delta^{1 - \epsilon}) = \Omega(\ln^2 \Delta)$, where $0 < \epsilon < 0.003$ is a constant. Combined with the known upper bound $\chi_{CN}(G) = O(\ln ^2 \Delta)$ for any graph $G$, due to \cite{bhyravarapu2021short}, we have $f_{CN}(c\Delta^{1 - \epsilon}) = \Theta(\ln^2 \Delta)$. In order to show that $f_{CN}(c\Delta^{1 - \epsilon}) = \Omega(\ln^2 \Delta)$, we need to show the existence of a graph with minimum degree $\Omega(\Delta^{1 - \epsilon})$ having CFCN chromatic number  $\Omega(\ln ^2 \Delta)$. We use the same random graph model used by Glebov, Szab{\'o}, and Tardos  in  \cite{glebov2014conflict} and show that such a graph exists with positive probability. Our proof is an extension of the proof of Theorem $4$ in \cite{glebov2014conflict} as it builds on the ideas presented there.

Let $A \subseteq V(G)$, for a graph $G$. We define $N^{(1)}_G (A) := \{v \in V(G) \backslash A : |N_G(v) \cap A| = 1\}$ to be the set of vertices outside $A$ that have exactly one neighbor in $A$.

\begin{theorem}
\label{thm:lower_bound}
 There exists a graph $G$ with maximum degree $\Delta $ and minimum degree $\Delta^{1 - \epsilon}$, where $0< \epsilon < 0.003$ is a constant, such that $\chi_{CN}(G) = \Omega(\ln^2 \Delta)$.   
\end{theorem}

Let $\epsilon_0 = \frac{\epsilon}{3}$.
Below, we describe the construction of a random graph $G$ on $n$ vertices. 
We use $V$ to denote $V(G)$. 
For the sake of simplicity, we assume that $\lfloor \ln n \rfloor$ divides $n$.
We partition the vertex set $V$ into parts $L_1,\dots, L_{\lfloor \ln n \rfloor}$ of size $\frac{n}{\lfloor \ln n \rfloor}$ each. We define the weight of a vertex $x \in L_i$ to be  $w_x = (1 - \epsilon_0)^i$.
For any $x \in L_i$, $y \in L_j$, we put an edge between $x$ and $y$ with probability $ w_x w_y =  (1 - \epsilon_0)^{i + j}$.
We define the weight of a set $S \subseteq V$ as, $$ w(S) = \sum_{v \in S}w_v.$$

Let $f: V(G) \longrightarrow [\epsilon_0^3 \ln^2 \Delta]$ be a coloring (not necessarily proper) of the vertices of $G$. We say that a vertex $x$ is \emph{taken care of} by a vertex $w$ under the coloring $f$ if
\begin{enumerate}
   \item $w \in N_G[x]$, and
    \item $f(w)$ is distinct from $f(y)$, for every $y \in N_G[x] \setminus \{w\}$. When $w = x$, we say that a vertex $x$ is \emph{taken care of by itself} under $f$.
\end{enumerate}

For each color class of the above coloring, the vertices that are taken care of by themselves form an independent set. The below lemma provides a bound on such vertices.

\begin{lemma}\label{lem:lb1}
    For the graph $G$ constructed, the independence number $\alpha(G) \leq n^{0.003}$
    asymptotically almost surely.
\end{lemma}

\begin{proof}
    We know that every edge in $G$ is present with probability at least $(1 - \epsilon_0)^{2 \ln n}$. Suppose $Pr[\alpha(G) > n^{0.003}]$ does not tend to 0 as 
    $n$ tends to infinity.
    Then, we can say that for a graph $H \in \mathcal{G}(n, (1 - \epsilon_0)^{2 \ln n})$, $Pr[\alpha(H) > n^{0.003}]$ too does not tend to 0 as $n \longrightarrow \infty$. 
    Here $\mathcal{G}(n, (1 - \epsilon_0)^{2 \ln n})$ denotes the Erd\H{o}s-R\'enyi
    graph on $n$ vertices where each edge is chosen with probability $(1 - \epsilon_0)^{2 \ln n}$.
    But this contradicts a known result (see Theorem 11.25 (ii) in Bollobás' book \cite{bollobas2001random}) that the largest independent set in $\mathcal{G}(n, p)$ has size at most $2\frac{\ln (n p)}{p}$  a.a.s. when $ 2.27/n \le p \le 1/2$. This would imply that the following holds a.a.s.:
\begin{eqnarray*}
    \alpha(H) & \le & 2\frac{\ln n}{(1 - \epsilon_0)^{2 \ln n}}  =\frac{2 \ln n}{n^{2 \ln (1 - \epsilon_0)}} = n^{\frac{\ln(2\ln n)}{\ln n} - 2 \ln (1 - \epsilon_0)} = n^{-2 \ln (1 - \epsilon_0) \Big[- \frac{\ln(2 \ln n)}{2 (\ln n)\cdot (\ln(1 - \epsilon_0))} +1\Big]} \\
    & = & n^{-2 \ln (1 - \epsilon_0) \big[1 - o(1)\big]} = n^{2 \ln \Big(\frac{1}{(1 - \epsilon_0)}\Big)\big[1 - o(1)\big]} \le n^{\ln \Big(\frac{1}{(1 - \epsilon_0)^2}\Big)} < n^{0.003}.
    \end{eqnarray*}
In the above, the last inequality follows since $\ln(1/(1-\epsilon_0)^2)$ increases 
when $\epsilon_0$ increases. By assumption, we have $\epsilon_0 < 0.001$, and hence 
$\ln(1/(1- \epsilon_0)^2) < \ln (1/0.999^2) < 0.003$.
\end{proof}
    
Let $S$ be a color class of the coloring $f$. Let $x \in L_i$ and suppose $x \notin S$. We shall use $p(x, S)$ to denote the probability that $x$ is taken care of by some vertex in the color class $S$. We have,

\begin{eqnarray*}
    p(x, S) = Pr[|N_G(x) \cap S| = 1] & = & \sum_{s \in S}Pr[N_G(x) \cap S = \{s\}] \\
    & = & \sum_{s \in S}w_s w_x  \prod_{y \in S \backslash \{s\}} (1- w_y w_x)  \\
    & < & w_x \sum_{s \in S}w_s \exp\Bigg(-\sum_{y \in S
    \backslash \{s\}} w_y w_x \Bigg) \\
    & = & w_x \sum_{s \in S}w_s \exp(- w(S) w_x + w_s w_x) \\
    & \le & w_x w(S) e^{-w_x w(S) + (1 - \epsilon_0)},
\end{eqnarray*}
where the third line follows since $1 - t \leq e^{-t}$. It can be verified that the function $ze^{- z}$ has a unique maximum at $z = 1$. Thus we get the following: 
\begin{equation}\label{eq:pxs}
p(x, S) < e^{- \epsilon_0}\;.    
\end{equation}

We say that a set  $S$ is \emph{heavy} if $w(S) > \sqrt{n}$; otherwise we call $S$ a \emph{light} set. Note that  any vertex has weight at least $(1 - \epsilon_0)^{\ln n} = n^{\ln (1 - \epsilon_0)} > n^{-2 \epsilon_0}$  (when $0 < \epsilon_0 < 0.5,$ we have that $ \ln(1 - \epsilon_0) \ge \frac{- \epsilon_0}{1 - \epsilon_0} > - 2\epsilon_0$). 
For any set 
 $S$, we have $w(S) > |S| \cdot n^{-2 \epsilon_0}$. Thus we have 
\begin{equation}\label{eq:lightbound}
    |S| < n^{0.5 + 2\epsilon_0}, \mbox{ when } S \mbox{ is a light set } (w(S) \leq \sqrt n). 
\end{equation}

The below lemma provides a bound on the number of vertices that are taken care of
by a heavy set. 
\begin{lemma}\label{lem:lb2}
For each heavy set $S \subseteq V$ of the graph $G$, asymptotically almost surely 
$|N^{(1)}(S)| < n^{0.6}$. That is, a.a.s., 
at most $n^{0.6}$ vertices that are not in $S$ are 
taken care of by $S$.
\end{lemma}

\begin{proof}
Consider a heavy subset $S \subseteq V$. Since $S$ is a heavy set, $w(S) > n^{0.5}$. Now fix a set $A \subseteq V \backslash S$ with  $|A| \ge n^{0.6}$. The probability that all elements $x \in A$ have exactly one neighbor in $S$ is estimated below.

\begin{eqnarray*}
    Pr[N^{(1)}(S) \supseteq A] & = & \prod_{x \in A} p(x, S) \\
    & \le & \prod_{x \in A} w_x w(S) e^{-w_x w(S) + 1}  \\
    & < & \Big(n^{(0.5 - 2\epsilon_0)} e^{(-n^{0.5 - 2\epsilon_0} + 1)}\Big)^{n^{0.6}} \\
    & = & \exp\Big((0.5 - 2\epsilon_0) \ln n - n^{0.5 -2\epsilon_0} + 1\Big)^{n^{0.6}} \\
    & = & \exp\Bigg(- n^{(1.1 - 2\epsilon_0)} \bigg(-\frac{(0.5 - 2\epsilon_0) n^{0.6} \ln n}{n^{(1.1 - 2\epsilon_0)}} + 1 - \frac{n^{0.6}}{n^{(1.1 - 2\epsilon_0)}} \bigg)\Bigg) \\
    & \le & \exp\bigg(-n^{(1.1 - 2\epsilon_0)} \Big(1 - o(1)\Big)\bigg).
\end{eqnarray*}
The inequality in the third line follows from the observation that $w_x w(S) > (1 - \epsilon_0)^{\ln n} \cdot n^{0.5} \ge n^{0.5 - 2\epsilon_0}$ and that $ze^{-z}$ is decreasing in the interval $[1,
\infty)$. Taking union over the possible $2^n \cdot 2^n$ choices for $S$ and $A$, we 
see that the probability of $S$ taking care of $A$ tends to 0.
\end{proof}

The next lemma bounds the number of vertices that can be 
taken care of by light sets. 

\begin{lemma}\label{lem:lb3}
Let $r = \lfloor \epsilon_{0}^3 \ln^2 n \rfloor$. For all pairwise disjoint light sets $S_1, \dots , S_r \subseteq V$, we have asymptotically almost surely $|\bigcup_{i=1}^{r} N^{(1)}(S_i)| < n - n^{0.7}$.    
\end{lemma}

\begin{proof}
We first fix light subsets $S_1, \dots, S_r$ of $V$. Since each $S_i$ is a light set, $w(S_i) \le \sqrt{n}$. 
We first need the following claim.

\begin{claim}
We have $\sum_{i=1}^{r} p(x, S_i) > \epsilon_{0} \ln n$ for at most half of the vertices $x \in V$.
\end{claim}

\begin{proof}
Assume the contrary. We then have
\begin{equation}\label{eqn:lbcalc1}
\frac{n}{2} \cdot \epsilon_{0} \ln n  \le  \sum_{x \in V} \sum_{i=1}^{r} p(x, S_i) 
     =  \sum_{i=1}^{r} \sum_{x \in V} p(x, S_i)\;. 
\end{equation}
When $S_i$ is fixed, we have seen that $p(x, S_i)\leq  w_x w(S_i) e^{-w_x w(S_i) + (1 - \epsilon_0)}$. When $x \in L_j$,  we get  
$$p(x, S_i)\leq  (1- \epsilon_0)^j w(S_i) e^{-(1- \epsilon_0)^j w(S_i) + 1}  = z_j e^{-z_j  + 1},$$ 
where the first inequality follows by dropping $-\epsilon_0$ from the exponent. We set  $z_j = (1- \epsilon_0)^j w(S_i)$ to get the second equality. Observe that 
\begin{equation}\label{eqn:lbcalc2}
\sum_{x \in V} p(x, S_i)  = \sum_{j = 1}^{\ln n} \sum_{x \in L_j} p(x, S_i)  \leq  \frac{n}{\ln n} \sum_{j=1}^{\ln n} z_j e^{-z_j + 1}  \leq  \frac{n}{\ln n} \sum_{j=1}^{\infty} z_j e^{-z_j + 1}.
\end{equation}
We will now upper bound $\sum_{j=1}^{\infty} z_j e^{-z_j + 1}$ by considering three ranges for $z_j$. Notice that $z_j >0$  when $S_i$ is nonempty.
\begin{itemize}
    \item When $0 < z_j \le 1$, we have $z_j e^{-z_j + 1} \le e z_j$. 
    Therefore, $\sum_{j: z_j \le 1} \big(z_j e^{-z_j + 1}\big) \leq \sum_{j: z_j \le 1} \big(e z_j\big) \le e \big( 1 + (1 - \epsilon_0) + (1 - \epsilon_0)^2 + \cdots\big) = \frac{e}{1 - (1 - \epsilon_0)} = \frac{e}{\epsilon_0}$.
    
    \item When $1 < z_j < 2$, we have $\sum_{j: 1 < z_j < 2} \big(z_j e^{-z_j + 1}\big)$
$ \le  \Big(\frac{e\cdot 2}{e^{2}} + \frac{e\cdot 2 (1 - \epsilon_0)}{e^{2 (1 - \epsilon_0)}} + \cdots +   \frac{e \cdot2 (1 - \epsilon_0)^{d-1}}{e^{2 (1 - \epsilon_0)^{d-1}}} \Big)$, where $d$ is the smallest integer value for which $2 (1 - \epsilon_0)^{d}$ goes below 1. 
It can be verified that each term of the above summation is at most 1. It can also be noted that the number of terms $d \leq \frac{1}{\epsilon_0}$. Thus we get that 
the terms when  $1 < z_j < 2$ sum to at most $\frac{1}{\epsilon_0}$.

    \item When $z_j \geq 2$, we have $\sum_{j: z_j \ge 2} \big(z_j e^{-z_j + 1}\big) \le \frac{1}{\epsilon_0} \int_{1}^{\infty} z e^{-z+1} \,dz = \frac{2}{\epsilon_0}$. The factor of 
    $\frac{1}{\epsilon_0}$ is due to the fact that $z_j - z_{j+1} = z_j - z_j (1 - \epsilon_0) = z_j \epsilon_0 > \epsilon_0$. Thus, the number of $z_j's$ that lie between $a$ and $a +1$, for any integer $a \ge 2$, is at most $\frac{1}{\epsilon_0}$. A straightforward integration by parts gives us 
    that $\int_{1}^{\infty} z e^{-z+1} \,dz= 2$. 
\end{itemize}
Combining equations (\ref{eqn:lbcalc1}), (\ref{eqn:lbcalc2}), and the above bound, gives us the following.

\begin{equation} \label{eqn:lbcalc5}
   \frac{n}{2} \cdot \epsilon_{0} \ln n \leq \sum_{i=1}^{r}  \frac{n}{\ln n} \sum_{j=1}^{\infty} z_j e^{-z_j + 1} \leq r \frac{n}{\ln n} \bigg(  \frac{e}{\epsilon_0} + \frac{1}{\epsilon_0} + \frac{2}{\epsilon_0}  \bigg) 
   \leq \frac{6rn}{\epsilon_0 \ln n}.
\end{equation}
Rearranging terms in inequality (\ref{eqn:lbcalc5}), we get $r \ge \frac{\epsilon_0^2 \ln^2 n}{12}$. This contradicts the fact that $r = \epsilon_0^3 \ln^2 n \le 0.001 \cdot \epsilon_0^2 \ln^2 n = \frac{\epsilon_0^2 \ln^2 n}{1000}$.

This completes the proof of the claim that  $ \sum_{i=1}^{r} p(x, S_i) > \epsilon_{0} \ln n$ for at most half the vertices $x \in V$.
\end{proof}



Let $V' \subseteq V$ be the set of those vertices $x \in V$ for which $ \sum_{i=1}^{r} p(x, S_i) \le \epsilon_0 \ln n$. Then, $|V'| \ge n/2$. Now fix a set $B\subseteq V$ with $|B| = n^{0.7}$. In the below calculation, we bound the probability that all the vertices  $x \in V \setminus B$ have exactly one neighbor in at least one of the $S_i$'s. This probability is given by
\begin{align*}
     \prod_{x \in V \setminus B}\Bigg(1- \prod_{i =1}^r \Big(1 - p(x, S_i) \Big) \Bigg) 
    & \le  \prod_{x \in V' \setminus B}\Bigg(1- \prod_{i =1}^r \Big(1 - p(x, S_i) \Big) \Bigg)  \\
    & \le  \exp\Bigg(- \sum_{x \in V'\setminus B} \prod_{i =1}^r \Big(1 - p(x, S_i) \Big)\Bigg)  \\
    & \le  \exp\Bigg(- \sum_{x \in V'\setminus B} e^{-f(\epsilon_0) \sum_{1 =1}^r p(x, S_i)}\Bigg) \text{, $f(\epsilon_0)$ is defined below.} \\
\end{align*}
The first inequality follows by restricting the scope of vertices, and the second inequality follows by using the fact that $1- x \leq e^{-x}$. For the  third inequality, let us set
$f(\epsilon_0) : = e^{\epsilon_0} \ln \big(\frac{1}{1 - e^{- \epsilon_0}}\big)$. 
From equation (\ref{eq:pxs}), we get that $p(x, S_i) \in [0, e^{-\epsilon_0}]$. When a number $z$ is chosen from the range $[0, 1)$, we can verify that $\big( \ln\big( \frac{1}{1 - z}\big) \big)/z$ is an increasing function on $z$. So $\big( \ln\big( \frac{1}{1 - e^{-\epsilon_0}}\big) \big)/ e^{-\epsilon_0} \geq \big( \ln\big( \frac{1}{1 - p(x,S_i) }\big) \big)/p(x,S_i)$. Rearranging this, it follows that $e^{-f(\epsilon_0) p(x, S_i)} \leq 1- p(x, S_i)$. We continue our computation below.
\begin{align}
    \exp\Bigg(- \sum_{x \in V'\setminus B} e^{-f(\epsilon_0) \sum_{1 =1}^r p(x, S_i)}\Bigg)
    & \le  \exp\Bigg(- \Big(\frac{n}{2} - n^{0.7}\Big) e^{-f(\epsilon_0)\epsilon_{0} \ln n}\Bigg) \nonumber\\
    & \le  \exp\Bigg(- \Big(\frac{n}{2} - n^{0.7}\Big) e^{-0.01 \ln n}\Bigg) 
    \nonumber \\
    & =  \exp\Big(- 0.5 n^{0.99} + n^{0.69}\Big) \nonumber \\
    & \le  \exp\Bigg(-n^{0.99}\Big(\frac{1}{2} - o(1) \Big)\Bigg) \;. \label{eq:n99}
\end{align}
The first inequality follows since for vertices $x$ of $V'$, we have $ \sum_{i=1}^{r} p(x, S_i) \le \epsilon_0 \ln n$. Below we explain  the second inequality. We have $f(\epsilon_0) \cdot \epsilon_0 = \epsilon_0 \cdot e^{\epsilon_0} \ln\big(\frac{1}{1 - e^{- \epsilon_0}}\big)$, where $0 < \epsilon_0 < 0.001$. Using $e^{-x} \le 1 - \frac{x}{2}$, for $x \in [0, 1]$, we get $f(\epsilon_0) \cdot \epsilon_0 \le \epsilon_0 \cdot e^{\epsilon_0} \ln\big(\frac{2}{\epsilon_0}\big)$. 
Since $x \cdot e^x \ln\big(\frac{2}{x}\big)$ is an increasing function when $x \in (0, 1]$ and considering that $\epsilon_0 \in (0, 0.001)$, we have, $f(\epsilon_0) \cdot \epsilon_0 \le 0.001 \cdot e^{0.001} \ln\big(\frac{2}{0.001}\big) < 0.01$.

Finally, since $|B| = n^{0.7}$ and using the size bound on light sets given in equation (\ref{eq:lightbound}), we note that the number of choices for sets $S_1, \dots, S_r$, and $B$ is at most

\begin{align*}
   \binom{n}{n^{0.7}} \binom{n}{n^{0.5 + 2\epsilon_0}}^r & <  n^{n^{0.7}} \Big(n^{n^{0.5 + 2\epsilon_0}} \Big)^r \\
   & =  e^{n^{0.7}\ln n} \Big(e^{(n^{0.5 + 2\epsilon_0})r \ln n} \Big) \\
   & =  e^{n^{0.7}\ln n + n^{(0.5 + 2\epsilon_0)} r\ln n} \\
   & =  e^{O(n^{0.7}\ln n)}.
\end{align*}
From  equation (\ref{eq:n99}) and the above calculations, we can see that any $r$ light sets $S_1, S_2, \ldots, S_r$
can take care of at most $n - n^{0.7}$ vertices a.a.s. 
\end{proof}

The below lemma shows that the graph $G$ constructed has the desired minimum degree with high probability.
\begin{lemma}\label{lem:lb4}
For the graph $G$ constructed, the minimum degree of $G$ is $\Omega(\Delta^{1 -  \epsilon})$ asymptotically almost surely.
\end{lemma}

\begin{proof}
We first calculate the expected degree of a vertex $x \in L_j$. Let $d_G(x)$ denote the degree of $x$. 
\begin{align}
    \mu(x) : = \mathbb{E}[d_G(x)] & =  \frac{n}{\ln n}\Big[(1 - \epsilon_0)^{j+1} + (1 - \epsilon_0)^{j+2} + \dots + (1 - \epsilon_0)^{j+\ln n} \Big] \nonumber \\
    & \geq  \frac{n}{\ln n} (1 - \epsilon_0)^{j+1} \label{eq:deglb} \\
    & = (1 - \epsilon_0) \cdot \frac{n}{\ln n} \cdot (1 - \epsilon_0)^{\ln (e^j)} \nonumber\\
    & = (1 - \epsilon_0) \cdot \frac{n}{\ln n} \cdot (e^j)^{\ln (1 - \epsilon_0)} \nonumber\\ 
    & \ge (1 - \epsilon_0) \cdot \frac{n}{\ln n} \cdot (e^j)^{- 2\epsilon_0}\;,\label{eq:deglblast}
\end{align}
where we use the fact that $\ln (1 - \epsilon_0) \ge - 2\epsilon_0 $, as noted in the discussion preceding the statement of Lemma \ref{lem:lb2}.
Using the Chernoff bound given in Theorem \ref{thm_Chernoff} for any $0 < \alpha < 1$, we get  
\begin{align*}
     Pr[|d_G(x) - \mu(x)| \ge \alpha \mu(x)] & \le  2e^{\frac{{-\alpha}^2\mu(x)}{3}} \\
    & \le  2e^{\frac{{-\alpha}^2 n}{\ln n} \cdot \frac{(1 - \epsilon_0)^{j + 1}}{3}}\;,
\end{align*}
where the last inequality follows by equation (\ref{eq:deglb}).
For an $x \in L_j$, let $A_{x}^j$ denote the event that $|d_G(x) - \mu(x)| \ge \alpha \mu(x)$. We have shown that $Pr[A_{x}^j] \le 2e^{\frac{{-\alpha}^2 n}{\ln n} \cdot \frac{(1 - \epsilon_0)^{j + 1}}{3}}$. Let $A^j$ denote the event $\bigcup_{x \in L_j}A_{x}^j$. Below we calculate the probability of $A^1$, using the union bound.
\begin{align*} 
    Pr[A^1] & \le  \sum_{x \in L_1} Pr[A_{x}^1]\\
    & \le  \frac{2n}{\ln n} \cdot e^{\frac{{-\alpha}^2 n}{\ln n} \cdot \frac{(1 - \epsilon_0)^{2}}{3}}  \\
    & \le 2n  \cdot \exp\bigg({-\frac{{\alpha}^2 n}{\ln n} \frac{0.999^2}{3}}\bigg) \hspace{3 cm} \text{(since $\epsilon_0 < 0.001$)} \\
    & \leq  \exp \bigg({\ln (2n)} - \frac{ \alpha^2 n }{4\ln n} \bigg) \;. \hspace{2.7 cm} \text{(since $0.999^2/3 > 1/4$)}
\end{align*}
Observe that, $Pr[A^{j +1}] \le e^{1 - \epsilon_0}  Pr[A^j] $, for  $ \forall 1 \le j < \ln n$. So we have the following for $1 \leq j \leq \ln n$.
\begin{equation*}
   Pr[A^j] \leq e^{j-1} \exp \bigg( {\ln (2n) } - \frac{ \alpha^2 n }{4\ln n} \bigg).
\end{equation*}
Let $\overline{A^j}$ denote the complement of the event $A^j$. We have 

\begin{eqnarray*}
    Pr[\overline{A^1} \cap \dots \cap \overline{A^{\ln n}}] & = & 1 - Pr[{A^1} \cup \dots \cup {A^{\ln n}}] \\
    & \ge & 1 - \Bigg[\frac{e^{\ln(2n)}}{e^{\frac{ \alpha^2 n }{4 \ln n}}} 
    + e \cdot \frac{e^{\ln(2n)}}{e^{\frac{ \alpha^2 n}{4 \ln n}}} + \cdots + e^{\ln n -1} \cdot \frac{e^{\ln(2n)}}{e^{\frac{ \alpha^2 n}{4 \ln n}}} \Bigg]\\
    & \geq & 1 - n \cdot  \Bigg[\frac{e^{\ln(2n)}}{e^{\frac{ \alpha^2 n }{4 \ln n}}} \Bigg] \\
    & = & 1 -o(1).
\end{eqnarray*}
We have thus shown that, for every vertex $x$ in $G$, $|d_G(x) - \mu(x)| < \alpha \mu(x)$ a.a.s. The vertex with the maximum degree is from $L_1$ a.a.s. and similarly, the vertex with the minimum degree is from $L_{\ln n}$ a.a.s. Now we calculate an upper bound for $\mu(x)$, for $x \in L_1$. 
\begin{eqnarray*}
    \mu(x) & = & \frac{n}{\ln n}\Big[(1 - \epsilon_0)^{1+1} + (1 - \epsilon_0)^{1+2} + \dots + (1 - \epsilon_0)^{1+\ln n} \Big] \\
    & \le & \frac{n}{\ln n} \frac{(1 - \epsilon_0)^{2}}{\epsilon_0} \;.
\end{eqnarray*}
From the above upper bound and the lower bound (computed in equation (\ref{eq:deglb})) on $\mu(x)$ when $x \in L_1$, we have $\Delta = \Theta\big( \frac{n}{\ln n}\big)$ a.a.s. Now, consider an $x \in L_{\ln n}$. From the lower bound for $\mu(x)$ that we computed at the beginning of this proof in equation (\ref{eq:deglblast}), we have, $\mu(x) \ge (1 - \epsilon_0) \cdot \frac{n}{\ln n} \cdot n^{-2 \epsilon_0}$. Thus, the minimum degree of $G$ is $\Omega\big( \frac{n^{1 - 2 \epsilon_0}}{\ln n}\big)$ a.a.s. That is, the minimum degree of $G$ is $\Omega(\Delta^{1 - \epsilon})$ a.a.s. (as $\epsilon = 3\epsilon_0$).
\end{proof}

To summarize the proof of Theorem \ref{thm:lower_bound}, we note that any coloring $f: V(G) \longrightarrow [r]$ that uses $r = \epsilon_0^3 \ln^2 \Delta$ 
many colors cannot take care of all the vertices of $G$. For a fixed color class, Lemma \ref{lem:lb1} bounds the number of vertices of that color class that can take care of themselves. 
Lemma \ref{lem:lb1} provides a bound of  $n^{0.003}$.
Thus across the $r$ color classes, 
the number of vertices that take care of themselves is at most  $r n^{0.003}$.
The number of vertices that are taken care of by a fixed heavy set is bounded by Lemma \ref{lem:lb2} 
to $n^{0.6}$. Thus the total number of vertices that are taken care of by heavy color classes is at most $r n^{0.6}$. 
Lemma \ref{lem:lb3} bounds the number of vertices taken care of by all the light color classes to $n - n^{0.7}$. Summing up, we note that all the $n$ vertices cannot be taken care of. 

Lemma \ref{lem:lb4} shows that the minimum degree of $G$ is $\Omega(\Delta^{1 - \epsilon})$, completing the proof of Theorem \ref{thm:lower_bound}. 




\bibliographystyle{plain}
\bibliography{Thesis_reference}
\end{document}